\documentclass{amsart}
\usepackage[]{latexsym,amssymb,amsmath,amsfonts, amsthm}
\usepackage[all,cmtip,ps]{xy}

\newtheorem{theorem}{Theorem}[section]

\newtheorem{corollary}[theorem]{Corollary}

\newtheorem{definition}[theorem]{Definition}

\newtheorem{lemma}[theorem]{Lemma}

\newtheorem{proposition}[theorem]{Proposition}
\theoremstyle{remark}
\newtheorem{remark}[theorem]{Remark}

\def\cX{{\mathcal X}}
\def\cT{{\mathcal T}}
\def\cF{{\mathcal F}}
\def\cY{{\mathcal Y}}
\def\cC{{\mathcal C}}\def\cD{{\mathcal D}}\def\cH{{\mathcal H}}\def\cS{{\mathcal S}}

\def\ra{{\rightarrow}}
\newcommand{\hH}{\mathrm{H}}

\def\cDl#1{{\cD^{\leq #1}}}
\def\cDg#1{{\cD^{\geq #1}}}

\newcommand{\la}{{\scriptscriptstyle\leftarrow}}
\newcommand{\point}{{\scriptscriptstyle\bullet}}

\newcommand{\id}{{\rm id}}
\newcommand{\Hom}{\mathrm{Hom}}
\newcommand{\End}{\mathrm{End}}
\newcommand{\Ker}{\mathrm{Ker}}
\newcommand{\Coker}{\mathrm{Coker}}

\newcommand{\Ima}{\mathrm{Im}}

\begin{document}
%\keywords{???} 
%\subclass{Primary ???} 
\title{On tilted Giraud subcategories}
%\subtitle { ?your subtitle ?} 
\dedicatory{Dedicated to Alberto Facchini on the occasion of his sixtieth birthday.}

\author{Riccardo Colpi, Luisa Fiorot, Francesco Mattiello}
\address[R. Colpi]{Dipartimento di Matematica Universit\`a degli Studi di Padova, Via Trieste 63, 35121 Padova}
\email{colpi@math.unipd.it}
\address[L. Fiorot]{Dipartimento di Matematica Universit\`a degli Studi di Padova, Via Trieste 63, 35121 Padova}
\email{fiorot@math.unipd.it}
\address[F. Mattiello]{Dipartimento di Matematica Universit\`a degli Studi di Padova, Via Trieste 63, 35121 Padova}
\email{mattiell@math.unipd.it}

\begin{abstract}
Firstly we provide a technique to move torsion pairs in abelian categories via adjoint functors and in particular
through Giraud subcategories. We apply this point in order to
develop a  correspondence between Giraud subcategories of an abelian category
$\cC$ and those of its tilt $\cH(\cC)$ i.e.,
the heart of a $t$-structure on $D^b(\cC)$ induced by a torsion pair.
\end{abstract}
\maketitle
\tableofcontents
%%%

\section*{Introduction}
One of the most useful process in Abelian category theory is the so-called localization of an abelian category $\cD$ 
to a quotient category $\cD/\cS$ by means of a Serre class $\cS$ in $\cD$. 
When $\cS$ is a localizing subcategory in the sense of \cite{0201.35602}, the canonical exact functor $\cD \to \cD/\cS$ has a fully faithful right adjoint functor 
$S: \cD/\cS \to \cD$ which allows to deal with $\cD/\cS$ as a full subcategory of $\cD$, which is called a Giraud subcategory of $\cD$. 
Dualizing the context, one get the notion of a co-Giraud subcategory.
Giraud and co-Giraud subcategories very often appear in the literature in very different settings (see \ref{rG}). 
%For example a well known result due to Popescu and Gabriel (see, for instance, \cite[Chapter 10]{MR0389953}) 
%tells that any Grothendieck category is in a natural way a Giraud subcategory of the category $R$-Mod of all the left $R$-modules, for a suitable ring R. 
%\sout{On the other hand, the Yoneda tensor-embedding $M\mapsto {}-\otimes_R M$ naturally makes $R$-Mod to be a co-Giraud subcategory of the Grothendieck category $(\rm{FP}_R, \rm{Ab})$ whose objects are the covariant functors from the finitely presented right R-modules to the abelian groups, and the morphisms are the natural transformations between them. 
%This allows, for instance, to deal with the extensively studied notion of pure-injective module by means of injective objects in $(\rm{FP}_R, \rm{Ab})$, thanks to a result of Gruson and Jensen \cite{MR633523}. Dually, the Yoneda Hom-embedding $M\mapsto \Hom_R(-,M)$ naturally makes $R$-Mod to be a Giraud subcategory of the Grothendieck category of contravariant functors $({}_R\rm{FP}^{op}, \rm{Ab})$. 
%Auslander proposed to study the representation theory of $R$ in terms of the ambient category $({}_R\rm{FP}^{op}, \rm{Ab})$, and in \cite{MR0335575} and \cite{MR0379599} he and Reiten studied deeper the subcategory of the finitely presented objects of $({}_R\rm{FP}^{op}, \rm{Ab})$, developing the powerful theory of almost split sequences for Artin algebras.}

On the other side, in 1981 Beilinson, Bernstein and Deligne introduced the notion of $t$-structure on a triangulated category
related to the study of the derived category of constructible sheaves on a stratified space. 
Actually the notion of $t$-structure  is a generalization of 
the notion of torsion pair on an abelian category (see for example  \cite{MR2327478}). 
In their work \cite{MR1327209} Happel, Reiten and Smalo related the study of
torsion pairs to Tilting theory and $t$-structures. 
In particular given an abelian category $\cC$ 
one can construct many non-trivial $t$-structures on its derived category
$D^b(\cC)$ by the procedure of tilting at a torsion pair (see \ref{Tstr}).

Inspired by the fundamental role of localizing subcategories in the study of 
problems of  gluing abelian categories or even triangulated categories we propose in this work a
bridge between the two previous abstract contexts.
The main progress in the present paper is to  show how the process of (co-) localizing moves from a basic abelian category to the level of its tilt, with respect to a torsion pair, and viceversa.

On the one side we deal with a (co-) Giraud subcategory $\cC$ of $\cD$, looking the way torsion pairs on $\cD$ reflect 
on $\cC$ and, conversely, torsion pairs on $\cC$ extend to $\cD$: in particular we find a one to one correspondence 
between arbitrary torsion pairs $(\cT, \cF)$ on $\cC$ and the torsion pairs $(\cX, \cY)$ on $\cD$ 
which are ``compatible'' with the (co-) localizing functor (Theorems~\ref{tt1-1} and \ref{ctt1-1}). 

On the other side, we compare this action of ``moving'' torsion pairs from $\cD$ to $\cC$ (and viceversa) with a ``tilting context'': 
more precisely, we look at the associated hearts $\cH_{\cD}$ and $\cH_{\cC}$ with respect to the torsion pairs $(\cT, \cF)$ on $\cC$ and $(\cX, \cY)$ on $\cD$, respectively, proving that $\cH_{\cC}$ is still a (co-) Giraud subcategory of $\cH_{\cD}$, and that the ``tilted'' torsion pairs in the two hearts are still related (Theorems~\ref{adjhearts} and \ref{cadjhearts}). Here the ambient Abelian category $\cD$ is arbitrary, with the unique request that the inclusion functor of $\cC$ into $\cD$ admits a right derived functor.

Finally given any Abelian category $\cD$ endowed with a torsion pair $(\cX, \cY)$, and considering any Giraud subcategory $\cC'$ of the associated heart 
$\cH_{\cD}$ which is ``compatible'' with the ``tilted'' torsion pair on $\cH_{\cD}$, we prove in Theorem~\ref{reconstruction}  how to recover 
a Giraud subcategory $\cC$ of $\cD$ such that
$\cC'$ is equivalent to the heart $\cH_{\cC}$ (with respect to the induced torsion pair).

%On the other hand, the Yoneda tensor-embedding $M\mapsto {}-\otimes_R M$ naturally makes $R$-Mod to be a co-Giraud subcategory of the Grothendieck category $(\rm{FP}_R, \rm{Ab})$ whose objects are the covariant functors from the finitely presented right R-modules to the abelian groups, and the morphisms are the natural transformations between them. 
%This allows, for instance, to deal with the extensively studied notion of pure-injective module by means of injective objects in $(\rm{FP}_R, \rm{Ab})$, thanks to a result of Gruson and Jensen \cite{MR633523}. Dually, the Yoneda Hom-embedding $M\mapsto \Hom_R(-,M)$ naturally makes $R$-Mod to be a Giraud subcategory of the Grothendieck category of contravariant functors $({}_R\rm{FP}^{op}, \rm{Ab})$. 
%Auslander proposed to study the representation theory of $R$ in terms of the ambient category $({}_R\rm{FP}^{op}, \rm{Ab})$, and in \cite{MR0335575} and \cite{MR0379599} he and Reiten studied deeper the subcategory of the finitely presented objects of $({}_R\rm{FP}^{op}, \rm{Ab})$, developing the powerful theory of almost split sequences for Artin algebras.

%% 
% if $R$ is an Artin algebra, then the natural duality induces an equivalence between  $(\FP_R, \Ab)$ and $({}_R\FP^{op}, \Ab)$.
%
%
%A different perspective in Abelian categories comes from the connection between torsion pairs and Tilting theory. Inspired by the work of Happel, Reiten and Smalo \cite{MR1327209}, 
%

%
%%%%%%%%%%%%%%%%%%%%%%%%%%%%%%%%%%%%%%%%%
\section{Serre, Giraud and co-Giraud subcategories} 
%%%%%%%%%%%%%%%%%%%%%%%%%%%%%%%%%%%%%%%%%

%%%
We begin by fixing some notations on Serre, Giraud and co-Giraud subcategories. A complete account on quotient categories and Serre classes can be found in \cite[Chapter 3]{0201.35602} and \cite[Section 1.11]{MR0102537}.

\begin{definition}
Let $\cD$ be an abelian category. A {\it Serre} class $\cS$ in $\cD$ is a full subcategory $\cS$ 
of $\cD$ such that for any short exact sequence $0\ra X_1\ra X_2\ra X_3\ra 0$ in $\cD$ the middle term
$X_2$ belongs to $\cS$ if and only if $X_1, X_3$ belong to $\cS$.
\end{definition}
%%%
\smallskip

The data of an abelian category $\cD$ and a Serre class $\cS$ of $\cD$ allow to construct a new abelian category, denoted by $\cD/\cS$, called the {\it quotient category of $\cD$ by $\cS$} (see \cite{MR0102537}). It turns out that $\cD/\cS$ is abelian and the canonical functor $T \colon\cD \to \cD/\cS$ is exact. A Serre class $\cS$ in $\cD$ is called a {\it localizing subcategory} (resp. {\it co-localizing subcategory}) if the functor $T$ admits a right adjoint (resp.~left adjoint) {\it section functor} $S$. In this case, the left exact (resp.~right exact) functor $S\circ T$ is called the {\it localization functor}. This localization functor is exact if and only if $S$ is exact (see \cite[Chapter 3]{0201.35602}).

%%%
\begin{definition}
An abelian category with a {\it distinguished Giraud subcategory} is the data $(\cD,\cC,l,i)$ of two abelian categories
$\cD$ and $\cC$ and two adjoint functors $\xymatrix{\cC \ar@<-0.5ex>[r]_{i}&\cD\ar@<-0.5ex>[l]_{l}}$ 
 (with $l$ left adjoint of $i$) such that  $l$ is exact and $i$ fully faithful. 
 
Dually an abelian category with a {\it distinguished co-Giraud subcategory} is the data $(\cD,\cC,j,r)$ of two abelian categories
$\cD$ and $\cC$ and two adjoint functors $\xymatrix{\cD \ar@<-0.5ex>[r]_{r}&\cC \ar@<-0.5ex>[l]_{j}}$ 
 (with $j$ left adjoint of $r$) such that  $r$ is exact and $j$ fully faithful.
\end{definition}
%%%
\smallskip

Therefore a localizing subcategory $\cS$ of $\cD$ defines a distinguished Giraud subcategory 
$(\cD, \cD/\cS, T, S)$. Conversely, given a distinguished Giraud subcategory $(\cD, \cC, l, i)$, the kernel of the functor $l$, i.e., the full subcategory $\cS$ of $\cD$ whose objects $S$ in $\cS$ satisfy $l(S)\cong0$, defines a localizing subcategory $\cS = \Ker(l)$ of $\cD$ whose associated quotient category is (equivalent to) $\cC$.

Let us denote by $\eta \colon \id_D \to i\circ l$ the unit of the adjunction $(l,i)$, and by $\cS^\bot$ the full subcategory of 
$\cD$ whose objects are defined by:
\[
\cS^\bot:=\{D\in\cD \,|\, \cD(S,D)=0, \forall S\in\cS\}.
\]
It turns out that 
\[
\cS^\bot=\{D\in\cD \,|\, \eta_D: D\ra il(D) \text{ is a monomorphism}\}.
\]
Moreover, let us notice that since $i$ is fully faithful the counit of the adjunction 
$\varepsilon :l\circ i\ra {\rm id}_{\cC}$ is an isomorphism of functors.

Dually, starting from a distinguished co-Giraud subcategory $(\cD, \cC, j, r)$, and denoting by 
$\varepsilon \colon j\circ r \to \id_D$ the counit of the adjunction $(j, r)$, the kernel of the functor $r$ defines a co-localizing subcategory $\cS = \Ker(r)$ of $D$ such that
\[\begin{matrix}
\;^\bot\cS:=&\{D\in\cD \,|\, \cD(D,S)=0, \forall S\in\cS\} \hfill \\
\hfill=&\{D\in\cD \,|\, \varepsilon_D: jr(D)\rightarrow D \text{ is an epimorphism}\}. \\
\end{matrix}
\]
Moreover, since $j$ is fully faithful, the unit of the adjunction 
$\eta :{\rm id}_{\cC} \ra r\circ j$ is an isomorphism of functors.
%%%
\smallskip

\begin{remark}\label{rG}
Giraud and co-Giraud subcategories very often appear in the literature in very different settings. 
For example a well known result due to Popescu and Gabriel (see, for instance, \cite[Chapter 10]{MR0389953}) 
tells that any Grothendieck category is in a natural way a Giraud subcategory of the category $R$-Mod of all the left $R$-modules, for a suitable ring R. 
On the other hand, the Yoneda tensor-embedding $M\mapsto {}-\otimes_R M$ naturally makes $R$-Mod to be a co-Giraud subcategory of the Grothendieck category $(\rm{FP}_R, \rm{Ab})$ whose objects are the covariant functors from the finitely presented right R-modules to the abelian groups, and the morphisms are the natural transformations between them. 
This allows, for instance, to deal with the extensively studied notion of pure-injective module by means of injective objects in $(\rm{FP}_R, \rm{Ab})$, thanks to a result of Gruson and Jensen \cite{MR633523}. Dually, the Yoneda Hom-embedding $M\mapsto \Hom_R(-,M)$ naturally makes $R$-Mod to be a Giraud subcategory of the Grothendieck category of contravariant functors $({}_R\rm{FP}^{op}, \rm{Ab})$. 
Auslander proposed to study the representation theory of $R$ in terms of the ambient category $({}_R\rm{FP}^{op}, \rm{Ab})$, and in \cite{MR0335575} and \cite{MR0379599} he and Reiten studied deeper the subcategory of the finitely presented objects of $({}_R\rm{FP}^{op}, \rm{Ab})$, developing the powerful theory of almost split sequences for Artin algebras.
\end{remark}
%%%
\smallskip

%%%%%%%%%%%%%%%%%%%%%%%%%%%%%%%%%%%%%%%%%
\section{Torsion and Torsion-free Classes} 
%%%%%%%%%%%%%%%%%%%%%%%%%%%%%%%%%%%%%%%%%

%%%
\begin{definition}
Given an abelian category $\cC$ a {\it torsion class} $\cT$ is a full subcategory of $\cC$ which is closed under taking inductive limits and extensions. Dually a  {\it torsion free class} $\cF$ is a full subcategory of $\cC$ which is closed under taking projective limits and extensions.
\par
A {\it torsion pair} $(\cT,\cF)$ in $\cC$ is a the data of a torsion class $\cT$ and a torsion free class $\cF$ such that $\cC(\cT,\cF)=0$ and any object $C\in\cC$ is the middle term of a short exact sequence $0 \to T \to C \to F \to 0$ with $T\in\cT$ and $F\in\cF$.
\end{definition}
%%%
\smallskip

A torsion class $\cT$ {\it cogenerates} $\cC$ when any object in $\cC$ is a subobject of a suitable object in $\cT$, and, dually, a torsion free class $\cF$ {\it generates} $\cC$ when any object in $\cC$ is a factor of a suitable object in $\cF$.
Typically, cogenerating torsion classes arise from Tilting theory and generating torsion free classes arise from Cotilting theory (see, for instance, \cite[Chapter I.3]{MR1327209} and \cite[Section 2]{MR2255195}).
%%%
\smallskip

%%%
\begin{remark}
If $\cC$ is a subcomplete abelian category in the sense of \cite{MR0191935} (that is, $\cC$ is an abelian category such that for any family $\{A_u \,|\, u\in U\}$ of subobjects of a fixed object $A$, the infinite sum 
$\sum_{u\in U}A_u$ and the infinite product $\prod_{u\in U}(A/A_u)$ exist in $\cC$), then any torsion class $\cT$ (torsion-free class $\cF$) on $\cC$ induces a torsion pair $(\cT,\cF)$ on $\cC$.
\end{remark}
%%%
\smallskip

The reader is referred to \cite[Chapter 1]{MR2327478} for more details.
\smallskip

In what follows, our aim is to move torsion class trough
exact functors  
and subsequently  trough a distinguished Giraud 
(resp. co-Giraud)
subcategory $\cC$ of $\cD$. 
Since torsion classes (resp.\ torsion free classes) 
are closed under inductive limits and extensions 
(resp.\ projective limits and extensions), 
it seems to us natural to use the left (resp.\ right) 
adjoint functor $l$ (resp.\ $i$), which respects inductive limits 
(resp.\ projective limits),
 in order to move torsion classes 
 (resp.\ torsion free classes) 
 from $\cC$ to $\cD$ (resp.\ from $\cD$ to $\cC$).
\bigskip

%%%
\begin{lemma}\label{MTC}(Dual to \ref{MTFC}).
Let $\cC$ be an abelian category and $\cT$ a torsion class on $\cC$.
Let $l:\cD \ra \cC$ be a functor between abelian categories which respects
inductive limits. Then the class
$$l^\la(\cT)=\{D\in\cD\; |\; l(D)\in \cT\}$$ 
is a torsion class in $\cD$. 
\end{lemma}
\begin{proof}
Clearly, the class $l^\la(\cT)$ is closed under taking inductive limits, because so is $\cT$ and $l$ respects inductive limits by assumption. Let us show that $l^\la(\cT)$ is closed under extensions. Consider a short exact sequence in $\cD$
\begin{equation*}
\xymatrix{
0 \ar[r] & X_1\ar[r] & D\ar[r] & X_2\ar[r] & 0}
\end{equation*}
with $X_1, X_2\in l^\la(\cT)$. By applying the functor $l$ (which is right exact) to this sequence we get an exact sequence in $\cC$
\[\xymatrix{
l(X_1)\ar[r] & l(D)\ar[r] & l(X_2)\ar[r] &0}
\]
with $l(X_1), l(X_2)\in\cT$. Taking the kernel $K$ of the morphism $l(D) \ra l(X_2)$, 
we see that $K$ is an epimorphic image of $l(X_1)$ and so $K\in\cT$, therefore $l(D)\in\cT$ as extension of objects in a torsion class. We conclude that $D\in l^\la(\cT)$.
\end{proof}
%%%
\smallskip

%%%
\begin{lemma}\label{MTFC}(Dual to \ref{MTC}).
Let $\cC$ be an abelian category and $\cF$ a torsion-free class on $\cC$.
Let $r:\cD \ra \cC$ be a functor between abelian categories which respects
projective limits. Then the class
$$r^\la(\cF)=\{D\in\cD\; |\; r(D)\in \cF\}$$ 
is a torsion-free class in $\cD$. 
\end{lemma}
%
%\begin{proof}DA MODIFICARE
%It is clear that $i^\la(\cY)$ is closed under projective limits, because so is $\cY$ and the functor $i$ is a right adjoint, so that it preserves inclusions and products. Let us show that $i^\la(\cY)$ is closed under extensions. Consider a short exact sequence
%\begin{equation}
%\xymatrix{
%0 \ar[r] & F_1\ar[r] & C\ar[r] & F_2\ar[r] & 0}
%\end{equation}
%with $F_1, F_2\in i^\la(\cY)$. By applying the functor $i$ to this sequence we get an exact sequence in $\cD$
%\[\xymatrix{
%0 \ar[r] & i(F_1)\ar[r] & i(C)\ar[r] & i(F_2)}
%\]
%with $i(F_1), i(F_2)\in\cY$. Taking the image $I$ of the morphism $i(C) \ra i(F_2)$, 
%we see that $I$ embeds in $i(F_2)$ and so $I\in\cY$, therefore $i(C)\in\cY$ as extension of objects in a torsion free class.  
%\end{proof}
%\smallskip

%%%%%%%%%%%%%%%%%%%%%%%%%%%%%%%%%%
\section{Moving Torsion Pairs trough Giraud subcategories}
%%%%%%%%%%%%%%%%%%%%%%%%%%%%%%%%%%

Given an abelian category $\cD$ with a distinguished Giraud subcategory $\cC$, by Lemma~\ref{MTC}, 
the class $l^\la(\cT):=\{D\in\cD \,|\, l(D)\in \cT\}$ is a torsion class on $\cD$.
\smallskip

%%%
\begin{proposition}\label{ttD}
Let  $\cD$  be an abelian category with a distinguished Giraud subcategory $\cC$. Suppose that $\cC$ is endowed with a torsion pair $(\cT,\cF)$.
Then the classes $(\hat{\cT},\hat{\cF})$:
$$\begin{matrix}
\hat{\cT}:= l^\la(\cT)=\{X\in\cD \,|\, l(X)\in\cT\} \hfill \\
\hat{\cF }:= l^\la(\cF)\cap\cS^\bot=\{Y\in {\cD} \,|\, Y\in{\cS^\bot} \text{ {\rm and} } l(Y)\in\cF \} \\
\end{matrix}
$$
define a torsion pair on $\cD$ such that $i(\cT)\subseteq \hat{\cT}$, $i(\cF)\subseteq \hat{\cF}$, 
$l(\hat{\cT})=\cT$, $l(\hat{\cF})=\cF$.
\end{proposition}
\begin{proof}
For any $T\in\cT$ we have $li(T)\cong T$, which proves that $i(\cT)\subseteq\hat{\cT}$.  
Moreover given $F\in\cF$ it is clear that $i(F)\in\cS^\bot$ and  $li(F)\cong F \in\cF$,
hence $i(\cF)\subseteq \hat{\cF}$. We deduce that $\cT=li(\cT)\subseteq l(\hat{\cT})\subseteq \cT$ and
$\cF=li(\cF)\subseteq l(\hat{\cF})\subseteq \cF$, which prove that 
$l(\hat{\cT})=\cT$ and $ l(\hat{\cF})=\cF$.
Let us show that $(\hat{\cT},\hat{\cF})$ is a torsion pair on $\cD$.

Given $X\in\hat{\cT}$ and $Y\in\hat{\cF}$, 
$$
{\cD}(X,Y)\hookrightarrow {\cD}(X,il(Y))\cong{\cC}(l(X),l(Y))=0.
$$

It remains to prove that for any $D$ in $\cD$ there exists a short exact sequence
\begin{equation*}\label{Dseq}
\xymatrix{
0 \ar[r] & X\ar[r] & D\ar[r] & Y\ar[r] & 0}
\end{equation*}
with $X\in\hat{\cT}$ and $Y\in\hat{\cF}$.

Given $D$ in $\cD$  there exist $T\in\cT$ and $F\in\cF$ such that
the sequence
\begin{equation}\label{l(D)seq}
\xymatrix{
0 \ar[r]& T\ar[r] & l(D) \ar[r] & F  \ar[r] & 0
}
\end{equation}
is exact.
Let define $X:=i(T)\times_{il(D)}D$; then we obtain the diagram 
\begin{equation}\label{1}
\xymatrix{
0 \ar[r]& i(T)\ar[r] & il(D) \ar[r] & i(F) &  \\
0 \ar[r] & X \ar[r]\ar[u] & D\ar[r]\ar[u]^{\eta_D} & {D/X}\ar@{^{(}->}[u] \ar[r]& 0 \\
}
\end{equation}
whose rows are exact (the first because the functor $i$ is left exact since it is a right adjoint,
while the second by definition) and the map
$D/X\hookrightarrow i(F)$ is injective since the first square is cartesian.

Let us apply the functor $l$ to (\ref{1})  remembering that $l$ is exact (so in particular
it preserves pullbacks and exact sequences)  and that
$l\circ i\cong {\rm id}_{\cC}$:
$$
\xymatrix{
0 \ar[r]& T\ar[r] & l(D) \ar[r] & F\ar[r] & 0 \\
0 \ar[r] & l(X) \ar[r]\ar[u]^{\cong} & l(D)\ar[r]\ar[u]_{{\rm id}_{l(D)}} & {l(D/X)}\ar[u]_{\cong} \ar[r]& 0. \\
}
$$
The first row coincides with (\ref{l(D)seq}) which is exact,
$l(X)\cong T\times_{l(D)}l(D)\cong T\in\cT$, which proves that $X\in\hat{\cT}$ and so
$l(D/X)\cong F\in\cF$, and the third vertical arrow of (\ref{1}) proves that $D/X\in\cS^\bot$, thus $D/X\in \hat{\cF}$.
\end{proof}
%%%
\smallskip

The following is a corollary of \ref{MTFC}:

%%%
\begin{corollary}\label{DCorGS}
Let  $\cD$  be an abelian category with a distinguished Giraud subcategory $\cC$. Suppose that $\cD$ is endowed with a torsion pair $(\cX,\cY)$.
Then the class
$i^\la(\cY):=\{C\in\cC \,|\, i(C)\in \cY\}$
is a torsion free class on $\cC$.
\end{corollary}
%%%
\smallskip

\begin{proposition}\label{GL}
Let  $\cD$  be an abelian category with a distinguished Giraud subcategory $\cC$. 
Suppose that $\cD$ is endowed with a torsion pair $(\cX,\cY)$,and let
\[
\begin{matrix}
l(\cX):=\{T\in\cC \,|\, T\cong l(X), \exists \,X\in\cX\} \hfill \\
l(\cY):=\{F\in\cC \,|\, F\cong l(Y), \exists \,Y\in\cY\}\\
\end{matrix}
\]
Then $(l(\cX),l(\cY))$ defines a torsion pair on $\cC$ if and only if $il(\cY) \subseteq \cY$. In this case, 
$i^\la(\cY)=l(\cY)$.
\end{proposition}
\begin{proof}
First let us suppose that $il(\cY) \subseteq \cY$. Then since $l\circ i\cong \id_\cC$ one has $i^\la(\cY)=l(\cY)$ and
by Corollary~\ref{MTFC} this is a torsion free class on $\cC$. 
Given $T \in l(\cX)$ (i.e., $T\cong l(X)$, with $X \in \cX$) and $F \in i^\la(\cY)$, one has 
${\cC}(X,F)={\cC}(l(X), F)\cong{\cD}(X,i(F))=0$, 
since $i(F)\in\cY$ by the definition of $i^\la(\cY)$. 
Now let $C\in \cC$. 
There exist $X\in\cX$, $Y\in\cY$ and a short exact sequence in $\cD$
\[\xymatrix{
0 \ar[r] & X\ar[r] & i(C)\ar[r] & Y\ar[r] & 0.}
\]
Applying the functor $l$ to the previous sequence we get a short exact sequence in $\cC$
\[\xymatrix{
0 \ar[r] & l(X)\ar[r] & C\ar[r] & l(Y)\ar[r] & 0}
\]
where $l(X)\in l(\cX)$ and $l(Y)\in l(\cY)$, which proves that $(l(\cX),l(\cY))$ is a torsion pair on $\cC$.

Conversely, if $(l(\cX),l(\cY))$ is a torsion pair on $\cC$ then for every $X \in \cX$ and every $Y\in \cY$ one has
$0=\cC(l(X),l(Y))\cong\cD(X,il(Y))$, therefore $il(Y)\in \cY$.
\end{proof}
%%%
\smallskip

>From \ref{ttD} and \ref{GL} we derive the following correspondence: 

\begin{theorem}\label{tt1-1}
Let  $\cD$  be an abelian category with a distinguished Giraud subcategory $\cC$.
There exists a one to one correspondence between  torsion pairs 
$(\cX,\cY)$ on $\cD$ satisfying  
$il(\cY)\subseteq\cY\subseteq \cS^\bot$ 
and torsion pairs $(\cT,\cF)$ on $\cC$.
\end{theorem}
%%%
\begin{proof}
From one side, taking a torsion pair $(\cT, \cF)$ in $\cC$, the torsion pair $(\hat{\cT}, \hat{\cF})$ 
satisfies $il(\hat{\cF})\subseteq \hat{\cF}$ and one can easily verify that
$(l(\hat{\cT}),l(\hat{\cF}))=(\cT,\cF)$.

On the other side given $(\cX,\cY)$  a torsion pair on $\cD$
satisfying  $il(\cY)\subseteq\cY\subseteq \cS^\bot$, its corresponding torsion pair on $\cC$ is
$(l(\cX),l(\cY))$ (by \ref{GL}) for whom
it is clear  that $\widehat{l(\cY)}:=l^\la(l(\cY))\cap \cS^\bot=\cY$ 
(since $\cY\subseteq \cS^\bot$)
and so $(\cX,\cY)=(\widehat{l(\cX)},\widehat{l(\cY)})$.
\end{proof}
%%%
\smallskip

%%%%%%%%%%%%%%%%%%%%%%%%%%%%%%%%%%
%\section{Torsion pairs trough co-Giraud subcategories}
%%%%%%%%%%%%%%%%%%%%%%%%%%%%%%%%%%

We briefly list 
%This section contains 
the statements dual in the case of co-Giraud subcategories whose
%to those of the previous section. It is clear that
proofs are simply the transcriptions of the previous ones in the opposite category. 
\smallskip
%%%%

First of all let us consider the following corollary of \ref{MTFC}:

%%%
\begin{corollary}\label{cCorcGS}
Let  $\cD$  be an abelian category with a distinguished co-Giraud subcategory $\cC$. 
Suppose that $\cF$ is  a torsion-free class in $\cC$.
Then the class
$r^\la(\cF):=\{D\in\cD \,|\, r(D)\in \cF\}$
is a torsion-free class on $\cD$.
\end{corollary}
%%%
\smallskip

%%%
\begin{proposition}\label{cttD}
Let  $\cD$  be an abelian category with a distinguished co-Giraud subcategory $\cC$. Suppose that $\cC$ is endowed with a torsion pair $(\cT,\cF)$.
Then the classes $(\hat{\cT},\hat{\cF})$:
$$\begin{matrix}
\hat{\cT}:= r^\la(\cT)\cap{}^\bot\cS=\{X\in\cD \,|\, X\in{}^\bot\cS \text{ {\rm and} } r(X)\in\cT\} \hfill \\
\hat{\cF }:= r^\la(\cF)=\{Y\in {\cD} \,|\,  r(Y)\in\cF \} \\
\end{matrix}
$$
define a torsion pair on $\cD$ such that $j(\cT)\subseteq \hat{\cT}$, $j(\cF)\subseteq \hat{\cF}$, 
$r(\hat{\cT})=\cT$, $r(\hat{\cF})=\cF$.
\end{proposition}
%%%
\smallskip

The following is a corollary of \ref{MTC}:

%%%
\begin{corollary}\label{cDCorGS}
Let  $\cD$  be an abelian category with a distinguished co-Giraud subcategory $\cC$. 
Suppose that $\cD$ is endowed with a torsion pair $(\cX,\cY)$.
Then the class
$j^\la(\cX):=\{C\in\cC \,|\, j(C)\in \cX\}$
is a torsion class on $\cC$.
\end{corollary}
%%%
\smallskip

\begin{proposition}\label{cGL}
Let  $\cD$  be an abelian category with a distinguished co-Giraud subcategory $\cC$. 
Suppose that $\cD$ is endowed with a torsion pair $(\cX,\cY)$, and let 
\[
\begin{matrix}
r(\cX):=\{T\in\cC \,|\, T\cong r(X), \exists \,X\in\cX\} \hfill \\
r(\cY):=\{F\in\cC \,|\, F\cong r(Y), \exists \,Y\in\cY\}\\
\end{matrix}
\]
Then $(r(\cX),r(\cY))$ defines a torsion pair on $\cC$ if and only if $jr(\cX) \subseteq \cX$. In this case, 
$j^\la(\cX)=r(\cX)$.
\end{proposition}
%%%
\smallskip
%%%

>From \ref{cttD} and \ref{cGL} we derive the following correspondence:

%%%
\begin{theorem}\label{ctt1-1}
Let  $\cD$  be an abelian category with a distinguished co-Giraud subcategory $\cC$.
There exists a one to one correspondence between  torsion pairs 
$(\cX,\cY)$ on $\cD$ satisfying  
$jr(\cX)\subseteq\cX\subseteq \;^\bot\cS$ 
and torsion pairs $(\cT,\cF)$ on $\cC$.
\end{theorem}
%%%
\smallskip

%%%%%%%%%%%%%%%%%%%%%%%%%%%%%%%%%%%%%%%%%
\section{$t$-structures induced by torsion pairs}
%%%%%%%%%%%%%%%%%%%%%%%%%%%%%%%%%%%%%%%%%

\begin{definition}\label{deft-s}%\cite[Definition ??]{MR751966}
A {\it $t$-structure} on a triangulated category $\cD$ is a pair $t=(\cDl0,\cDg0)$ of strictly full subcategories of $\cD$ such that, 
setting $\cDl n:=\cDl0[-n]$ and $\cDg n:=\cDg0[-n]$, one has
\begin{enumerate}
\item[\rm (0)] $\cDl0\subseteq\cDl1$ and $\cDg0\supseteq\cDg1$.
\item[\rm (i)] $\cD(X,Y)=0$ for every $X$ in $\cDl0$ and every $Y$ in $\cDg1$.
\item[\rm (ii)] For any object $X\in\cD$ there exists a distinguished triangle: 
\[A\to X\to B \to A[1]
\]
in $\cD$ such that
$A\in\cDl0$ 
and $B\in\cDg1$.
\end{enumerate}
\end{definition}
%%%
\smallskip
%%%
\begin{proposition}{\rm\cite[Proposition 1.3.3]{MR751966}}\label{t-adj}
Let ${\rm t}=(\cDl0,\cDg0)$ be a $t$-structure on a triangulated category $\cD$.
\begin{enumerate}
\item[\rm (i)]The inclusion of $\cDl n$ in $\cD$ admits a right adjoint $\tau^{\leq n}$, and the inclusion of $\cDg n$ in $\cD$ a left adjoint $\tau^{\geq n}$,
 called the truncation functors.
\item[\rm (ii)]For every $X$ in $\cD$ there exists a unique morphism 
$d\colon \tau^{\geq 1}(X)\to \tau^{\leq 0}(X)[1]$ such that the triangle
\[
\tau^{\leq 0}(X)\ra X\ra \tau^{\geq 1}(X)\overset{d}\ra
\]
is distinguished. This triangle is (up to a unique isomorphism) the unique distinguished triangle 
$(A,X,B)$ with $A$ in $\cDl0$ and $B$ in $\cDg1$.
\item[\rm (iii)]The category $\cH_t:=\cDl0\cap\cDg0$ is abelian, and the truncation functors induce a functor 
$\hH_t\colon \cD \to \cH_t$, called the $t$-cohomological functor 
($H_t^0(X)=\tau^{\geq0}\tau^{\leq0}(X)\cong\tau^{\leq0}\tau^{\geq0}(X)$ and for every $i\in \Bbb Z$, 
$\hH_t^i(X)=\hH_t^0(X[i])$, see {\rm\cite[Theorem~1.3.6]{MR751966}}).
\end{enumerate}
\end{proposition}
%%%
\smallskip
In particular, given an abelian category $\cC$ its (unbounded) derived category $D(\cC)$ is a triangulated category 
which admits a canonical $t$-structure, called the {\it natural $t$-structure}, whose class $D(\cC)^{\leq 0}$ 
(resp. $D(\cC)^{\geq 0}$) is that of complexes without cohomology in positive (resp. negative) degrees.
We will denote by $\hH\colon D(\cC) \to \cC$ its cohomological functor and by
$t^{\leq n}$ resp. $t^{\geq n}$ its truncation functors.
As explained by A.~Beligiannis and I.~Reiten in their work \cite{MR2327478},
one can regard a $t$-structure on a triangulated category $\cD$ as a generalization
of a torsion pair, where the role of the torsion class is provided by $\cDl0$,
while that of the torsion free class is played by $\cDg1$.
Moreover, given a torsion pair on an abelian category $\cC$ one can construct a 
$t$-structure on its derived category $D(\cC)$, as explained in  \cite{MR1327209}.
\endgraf
Let us briefly recover this construction:
%%%
%%%
%%%%%%
%%%
\begin{proposition}
Let  $(\cT,\cF)$  be a  torsion pair on an abelian category $\cC$.
The classes
$$\begin{matrix}
t(\cT)=\cD_{\rm t}^{\leq 0}= & \{ C^\point\in D(\cC)\; | \; H^0(C^\point)\in\cT,\; H^i(C^\point)=0 \;  \forall i>0 \} \hfill\\
t(\cF)=\cD_{\rm t}^{\geq 0}= & \{C^\point\in D(\cC)\; | \; H^{-1}(C^\point)\in\cF,\; H^i(C^\point)=0 \;  \forall i<-1 \} \\
\end{matrix}
$$
define a $t$-structure on $D(\cC)$ which is called the  $t$-structure induced by the torsion pair $t$. 
\end{proposition}
\begin{proof}
It is straightforward to verify condition ($0$) of definition \ref{deft-s}.
Let us show that condition (i) holds. Indeed, given
$X^\point \in \cD_{\rm t}^{\leq 0}$ and $Y^\point\in \cD_{\rm t}^{\geq 1}$ (i.e.,$H^0(Y^\point)\in\cF$, and 
$H^i(Y^\point)=0 \;  \forall i<0$) one has
\begin{eqnarray*}
D(\cC)(X^\point,Y^\point)\cong & D(\cC)\left(\tau^{\geq 0}(X^\point),Y^\point\right)\cong 
D(\cC)\left(\tau^{\geq 0}(X^\point),\tau^{\leq 0}(Y^\point)\right)\cong \hfill \cr
\hfill \cong & D(\cC)\left(H^0(X^\point)[0],H^0(Y^\point)[0]\right)\cong \cC\left(H^0(X^\point),H^0(Y^\point)\right)=0 \hfill \cr
\end{eqnarray*}
Finally, let us prove condition (ii). Given any $C^\point\in D(\cC)$, let us consider the object $T=t(H^0(C^\point))$ which is the torsion part of the zero cohomology 
$H^0(C^\point)={{\Ker}(d^0_C)\over{{\Ima}(d^{-1}_C)} }$,
and let us define $X$ to be the fiber product 
\[
X=T\times_{H^0(C^\point)}{\Ker}(d^0_C).
\]
Then we obtain the following functorial construction:
\\
\[
\xymatrix@C=1.7em{
\Ima(d_C^{-1}) \ar@{^{(}->}@/^2pc/[rr]\ar[rd]_{0} \ar@{^{(}.>}[r]^(0.6){\exists !} 
& X \ar@{^{(}->}[r]\ar@{->>}[d] & {\Ker}(d^0_C) \ar@{->>}[d] \\
& T\ar@{^{(}->}[r] & H^0(C^\point)\\
}
\]
where ${X\over{\Ima(d^{-1}_C)}}\cong T$. 
This permits to define the short exact sequence of complexes:
\[
\xymatrix@C=1.7em{
0= \ar[d] & 
[\cdots \ar[r]  & 0 \ar[r]\ar[d] &
 0    \ar[r]\ar[d]  & 0  \ar[r]\ar[d] & 0  \ar[r]\ar[d] &   0\ar[r]\ar[d] &
 \cdots ] \hfill \\
\tau_t^{\leq 0}(C^\point)= \ar[d] &
[\cdots \ar[d]\ar[r]  & C^{-2}\ar[d]^{\cong} \ar[r]^{d_C^{-2}} &
 C^{-1}  \ar[d]^{\cong}  \ar[r]^{d_C^{-1}}  & X \ar[d] \ar[r] & 0  \ar[d]\ar[r] &   0\ar[d]\ar[r] &
 \cdots ] \hfill \\
{\rm id}(C^\point)= \ar[d]& [\cdots \ar[r]\ar[d]  & C^{-2}\ar[d] \ar[r]^{d_C^{-2}} &
 C^{-1}  \ar[d]  \ar[r]^{d_C^{-1}}  & C^0\ar[d]\ar[r]^{d^0_C} & C^1\ar[d]^{\cong}\ar[r]^{d^1_C} & C^2 \ar[d]^{\cong}\ar[r] &   
 \cdots ] \hfill \\
\tau_t^{\geq 1}(C^\point)= \ar[d]&
[\cdots \ar[r]  & 0 \ar[r]\ar[d] & 0\ar[r]\ar[d]  & {C^0\over{X}} \ar[r]^{d^0_C}\ar[d] & C^1\ar[r]^{d^1_C}\ar[d] & C^2 \ar[r] \ar[d]&   
 \cdots ] \hfill \\
 0=  & 
[\cdots \ar[r]  & 0 \ar[r] &
 0    \ar[r]  & 0  \ar[r] & 0  \ar[r]&   0\ar[r] &
 \cdots ] \hfill \\
}
\]
such that 
$\tau_t^{\leq 0}(C^\point)\in \cD_{\rm t}^{\leq 0}$ 
(since 
$H^0(\tau_t^{\leq 0}(C^\point))={X\over{\Ima(d^{-1}_C)}}\cong T\in \cT$ and $H^i(\tau_t^{\leq 0}(C^\point))=0$
for any $i>0$)
and
$\tau_t^{\geq 1}(C^\point)\in \cD_{\rm t}^{\geq 1}$  
(since 
$H^0(\tau_t^{\geq 1}(C^\point))={{\Ker}(d^0_C)\over{X}}\cong {H^0(C)\over T}\in \cF$ and $H^i(\tau_t^{\leq 0}(C^\point))=0$
for any $i<0$).
\endgraf
Let us recall that any short exact sequence of complexes in an abelian category $\cC$ induces a distinguished triangle
in its derived category $D(\cC)$ (see \cite[Section 2.4.2]{MR2286904}). 
In our case the previous exact sequence provides a distinguished triangle
$\tau_t^{\leq 0}(C^\point)\to C^\point\to \tau_t^{\geq 1}(C^\point) \to \tau_t^{\leq 0}(C^\point)[1]$,
and this concludes the proof.
\end{proof}
\smallskip
%%%
\begin{remark}
Let $\cC$ be an abelian category endowed with the trivial torsion pair
$(\cC, 0)$. The the $t$-structure associated to this trivial torsion pair is the trivial $t$-strucuture
on $D(\cC)$ (see \ref{t-adj}).
\end{remark}
%%%
\smallskip
\begin{remark}\label{Tstr}\label{RTS}
Let $\cC$ be an abelian category with  a torsion pair $(\cT,\cF)$. 
The  {\it heart} associated to the $t$-structure
$(\cD_{\rm t}^{\leq 0},\cD_{\rm t}^{\geq 0})$ on $D(\cC)$
is the full subcategory $\cH_{\cC} := t(\cT)\cap t(\cF)$ of $D(\cC)$ called the tilt of $\cC$ by the torsion pair
$(\cT,\cF)$. 
It is shown in 
\cite{MR751966} that $\cH_{\cC}$ is an abelian category where short exact sequences are deduced by distinguished triangles in $D(\cC)$.
The objects of $\cH_{\cC}$ are represented, up to isomorphism, by complexes of the form
\begin{equation*}
X:\ X^{-1}\overset{x}{\longrightarrow }X^{0},\;\text{ with }\Ker(x)\in \cF\text{ and }
\Coker(x)\in \cT,
\end{equation*}
while a morphism $\phi : X \ra Y$ in $\cH_{\cC}$ is a formal fraction $\phi = (s)^{-1}\circ f$, where:
\begin{enumerate}
\item $X\overset{f}{\longrightarrow}Z$ is a representative of a homotopy class of maps of complexes 
$$
\xymatrix{X^{-1}\ar[d]_{f^{-1}}\ar[r]^x & X^0\ar[d]^{f^0} \\
Z^{-1}\ar[r]^z & Z^0
}
$$
where we recall that $X\overset{f}{\longrightarrow }Z$ is null-homotopic if there is
a map $r^{0}:X^{0}\rightarrow Z^{-1}$ such that%
\begin{equation*}
f^{0}=zr^{0}\text{ \ and \ }f^{-1}=r^{0}x
\end{equation*}
\item $Y\overset{s}{\longrightarrow }Z$ is a quasi-isomorphism, i.e., it is a map of complexes 
which induces isomorphism in cohomology:
$$
\xymatrix{0\ar[r] & \Ker(y) \ar[d]_{\cong} \ar[r] & Y^{-1}\ar[d]_{s^{-1}} \ar[r]^y & Y^0 \ar[d]^{s^0} \ar[r] & \Coker(y)\ar[d]^{\cong} \ar[r] & 0 \\
0\ar[r] & \Ker(z) \ar[r] & Z^{-1} \ar[r]^z & Z^0  \ar[r] & \Coker(z) \ar[r] & 0
}
$$
\end{enumerate}
%%%
%\smallskip
Every distinguished triangle $X_1^\point\ra X_2^\point \ra X_3^\point \buildrel{[+1]}\over\ra X_1^\point[1]$ 
 in $D(\cC)$ provides a long exact sequence of $t$-cohomology
in the heart $\cH_\cC$:
\begin{equation*}
\cdots H_t^{-1}(X_3)\ra H_t^0(X_1)\ra H_t^0(X_2)\ra H_t^0(X_3)\ra 
H_t^1(X_1)\cdots
\end{equation*}
Moreover given an object $C$ in $\cC$, its $t$-cohomology objects in $\cH_\cC$ are
$H^i_t(C)=0$ for any $t<0$, $t>1$; $H^0_t(C)=t(C)[0]$ is the torsion part of $C$ (with respect
to the torsion pair $(\cT,\cF)$) placed in degree zero, while $H^1_t(C)=\frac{C}{t(C)}[1]$.
The tilted pair $(\cF[1],\cT[0])$ is a torsion pair in $\cH_{\cC}$ with category equivalences 
$\cF[1]\cong\cF$ and $\cT[0]\cong\cT$ (see \cite[Corollary~2.2]{MR1327209}).
\end{remark}
\smallskip
%%%

\begin{remark}
In \cite{MR2255195} the authors introduced the notion of a tilting object for an arbitrary Abelian category, proving that for any ring 
$R$ and for any faithful torsion pair $(\cX, \cY)$ in $R$-Mod the heart $\cH(\cX, \cY)$ of the t-structure in 
$D(R)$ associated to $(\cX, \cY)$ is (an abelian category) with a tilting object $T=R[1]$. 
Then, again the first author with Gregorio and Mantese in \cite{MR2275375} showed that the heart is a prototype for these categories, in the sense that an Abelian category 
$\cD$ admits a tilting object $T$ if and only if $\cD$ is equivalent to the category $\cH(\cX, \cY)$ for a suitable torsion pair $(\cX, \cY)$ in $\End(T)$-Mod which is ``tilted'' by $T$, and with Gregorio in \cite{PREPRINT} proved that $\cH(\cX, \cY)$ is a Grothendieck category if and only if the torsion pair is cogenerated by a cotilting module in the sense of \cite{MR1448804}.
This allows us to deal with a more general notion of a``tilting context'': given an Abelian category $\cD$ endowed with a faithful torsion pair $(\cX, \cY)$ (i.e., such that $\cY$ generates $\cD$), we get a new Abelian category $\cH(\cX, \cY)$ endowed with a torsion pair $(\cY[1], \cX[0])$ which is ``tilting'', in the sense that the torsion class $\cY[1]$ cogenerates the category $\cH(\cX, \cY)$ and there are category equivalences $\cY[1]\cong\cY$ and $\cX[0]\cong\cX$ induced by exact functors.
\end{remark}
\smallskip
%%%%
Let $\cC$ be an abelian category endowed with a torsion pair $(\cT,\cF)$.
Since we need to use different torsion pairs we would use the notation
$(t(\cT),t(\cF))$ instead of $(\cD_{\rm t}^{\leq 0},\cD_{\rm t}^{\geq 0})$ to denote the $t$-structure
associated to the torsion pair $(\cT,\cF)$.
For the same reason when we need to 
clarify the torsion pair we  would denote by $\tau_{t(\cT)}, \tau_{t(\cF)}$ the truncation functors
instead of $\tau_t^{\leq 0},\tau_t^{\geq 1}$.
%%%
\smallskip

As showed in \cite{MR2327478}, there exists an injective function between the poset of torsion pairs in $\cC$ and that of $t$-structures in $D(\cC)$.
Moreover one can recover those $t$-structures on $D(\cC)$ which are induced by torsion pairs by means of the following fact proved by Keller and Vossieck in \cite{0663.18005}, and Polishchuk in \cite[Lemma 1.2.2]{MR2324559}.
\smallskip
%%%%

\begin{theorem}
Given $\cC$ an abelian category. There exists a bijection between
\begin{enumerate}
\item torsion pairs on $\cC$
\item $t$-structures $(\cT^{\leq 0},\cT^{\geq 0})$ on $D(\cC)$ such that
$D(\cC)^{\leq -1}\subset \cT^{\leq 0}\subset D(\cC)^{\leq 0}$.
\end{enumerate}
\end{theorem}
\begin{proof}
We have just seen that any torsion pair $(\cT,\cF)$ induces
the $t$-structure $(t(\cT),t(\cF)[1])$ on $D(\cC)$ which, by definition, satisfies
$D(\cC)^{\leq -1}\subset t(\cT)\subset D(\cC)^{\leq 0}$.
\\
On the other side given $(\cT^{\leq 0},\cT^{\geq 0})$ a 
$t$-structure on $D(\cC)$ such that
$D(\cC)^{\leq -1}\subset \cT^{\leq 0}\subset D(\cC)^{\leq 0}$ we obtain (by orthogonality) that
$D(\cC)^{\geq 1}\subset \cT^{\geq 1}\subset D(\cC)^{\geq 0}$.
Then the classes $\cT=\cC[0]\cap \cT^{\leq 0}$ and
$\cF=\cC[0]\cap \cT^{\geq 1}$ define a torsion pair in $\cC$
whose approximation short exact sequence for an object $C\in\cC$
is given by the long exact sequence of $\hH$ cohomology for the distinguished triangle $\tau_{t(\cT)}(C)\ra C\ra  \tau_{t(\cF)}(C)$
since $\tau_{t(\cT)}(C)\in \subset D(\cC)^{\leq 0}$ (resp. $\tau_{t(\cF)}(C)\in \subset D(\cC)^{\geq 0}$).
\end{proof}

%%%
\smallskip

%%%%%%%%%%%%%%%%%%%%%%%%%%%%%%%%%%%%%%%%%
\section{Tilted Giraud subcategories}
%%%%%%%%%%%%%%%%%%%%%%%%%%%%%%%%%%%%%%%%%

\begin{lemma}\label{DlComm}
Let $\cD$ and $\cC$ be abelian categories and $l \colon \cD \to \cC$ be an exact functor. Suppose that $\cD$ is endowed with a torsion pair $(\cX,\cY)$ and that $(\cT,\cF)=(l(\cX),l(\cY))$ defines a torsion pair on $\cC$. Then
$Dl\circ\tau_{t(\cY)}=\tau_{t(\cF)}\circ Dl$ and $Dl\circ\tau_{t(\cX)}=\tau_{t(\cT)}\circ Dl$. In particular, $Dl$ commutes with the functors $H^0_t$.
\end{lemma}
\begin{proof}
Since $l$ is exact it admits a total derived functor $Dl \colon D(\cD) \to D(\cC)$. Moreover, from $l(\cX)=\cT$ and $l(\cY)=\cF$ we derive that $Dl(t(\cX))\subseteq t(\cT)$ and $Dl(t(\cY))\subseteq t(\cF)$, i.e.,$Dl$ is an exact functor for the $t$-structure $(t(\cX),t(\cY))$ on $D(\cD)$ (see \cite[1.3.16]{MR751966}).
Let $D^\point\in D(\cC)$ and 
\begin{equation}\label{triaD}
\xymatrix{ \tau_{t(\cX)}(D^\point) \ar[r]& D^\point \ar[r]& \tau_{t(\cY)}(D^\point) \ar[r]^{\qquad+1}& 
}
\end{equation}

its distinguished triangle, with $\tau_{t(\cX)}(D^\point) \in t(\cX)$ and $\tau_{t(\cY)}(D^\point) \in t(\cY)$. By applying the functor $Dl$ to (\ref{triaD}) we get the triangle in $D(\cC)$

\begin{equation}\label{triaC}
\xymatrix{ Dl(\tau_{t(\cX)}(D^\point)) \ar[r]& Dl(D^\point) \ar[r]& Dl(\tau_{t(\cY)}(D^\point)) \ar[r]^{\qquad+1}& 
}
\end{equation}

with $Dl(\tau_{t(\cX)}(D^\point)) \in t(\cT)$ and $Dl(\tau_{t(\cY)}(D^\point)) \in t(\cF)$, so (\ref{triaC}) is the distinguished triangle associated to $Dl(D^\point)$, which proves that $Dl\circ\tau_{t(\cY)}=\tau_{t(\cF)}\circ Dl$ and $Dl\circ\tau_{t(\cX)}=\tau_{t(\cT)}\circ Dl$.
\end{proof}
%%%
\smallskip

\begin{proposition}\label{SonH}
Let $\cS$ be a Serre subclass in an abelian category $\cD$, and suppose that $\cD$ is endowed with a torsion pair
$(\cX,\cY)$ such that $(l(\cX),l(\cY))$ is a torsion pair on the quotient category
$\cC :=\cD/\cS$.
Then  $l$ induces  a functor $l_\cH \colon \cH_{\cD} \to \cH_{\cC}$ on the associated hearts which 
is exact and essentially surjective and so, denoted by $\cS_\cH$ the kernel of $l_\cH$, one has that $\cS_\cH$
is a Serre subclass of $\cH_\cD$ and $\cH_\cC\cong \cH_\cD/\cS_\cH$.
\end{proposition}
\begin{proof}
As $l$ is exact and it respects the torsion pairs
its total derived functor $Dl \colon D(\cD) \to D(\cC)$ 
is exact with respect to the $t$-structures associated to the torsion pairs so the 
the restriction of $Dl$ to $\cH_{\cD}$  defines a functor
$l_\cH \colon \cH_{\cD} \to \cH_{\cC}$  on the hearts which is  exact.
In particular the kernel $\cS_\cH$ of $l_\cH$ is a Serre subclass of $\cH_\cD$.
The functor $l_\cH$ is essentially surjective since given an object $X^\point \in \cH_{\cC}$, 
there exists $D^\point \in \cH_{\cD}$ such that $X^\point \cong Dl(D^\point)$. 
Therefore using Lemma~\ref{DlComm} we find that 
\[
\begin{matrix}
X^\point =& H^0_t(X^\point) 
\hfill\cong& H^0_t\circ Dl(D^\point) \hfill \\
\hfill\cong& Dl\circ H^0_t(D^\point)  \cong& l_{\cH}\circ H^0_t(D^\point), \hfill \\
\end{matrix}
\]
which proves that $l_{\cH}$ is essentially surjective, so that
by \cite[Chapter 3, Section 1, Corollary~2]{0201.35602} we get
$\cH_\cC\cong \cH_\cD/\cS_\cH$.
\end{proof}
%%%
\smallskip

\begin{theorem}\label{adjhearts}
Let  $\cD$  be an abelian category with a distinguished Giraud subcategory $\cC$ such that 
$i$ admits a right derived functor ${Ri}$. 
Let $(\cX,\cY)$ be a torsion pair on $\cD$ such that $il(\cY)\subseteq\cY$, and let $(\cT,\cF)=(l(\cX),l(\cY))$ be the induced torsion pair on $\cC$.
 Let us denote by $\cH_{\cC}$ and $\cH_{\cD}$ the associated hearts. 
Then there exists a distinguished Giraud subcategory
$(\cH_\cD,\cH_\cC,l_\cH,i_\cH)$ such that 
$i_\cH(l_\cH(\cX[0]))\subseteq \cX[0]$.
\end{theorem}
\begin{proof}First we remark that since $l$ and $i$ are additive, they extend to an adjunction $\xymatrix{K(\cC) \ar@<-0.5ex>[r]_{i}& K(\cD)\ar@<-0.5ex>[l]_{l}}$ between the homotopy categories. Moreover, since $l$ is exact it admits a total derived functor $Dl\colon D(\cD)\to   D(\cC)$. Therefore $\xymatrix{D(\cC) \ar@<-0.5ex>[r]_{Ri}&D(\cD)\ar@<-0.5ex>[l]_{Dl}}$ are two adjoint functors (with $Dl$ left adjoint of $Ri$) by \cite[Section 3.1]{MR2384608}, and $Dl\circ Ri\cong R(l\circ i)\cong id_{D(\cC)}$.

By Proposition~\ref{SonH}, $l$ induces  a functor $l_\cH \colon \cH_{\cD} \to \cH_{\cC}$ on the associated hearts which 
is exact and essentially surjective and so $\cH_\cC\cong \cH_\cD/\cS_\cH$.

%Since $l(\cX)=\cT$ and $l(\cY)=\cF$, we derive that $Dl(t(\cX))\subseteq t(\cT)$ and $Dl(t(\cY))\subseteq t(\cF)$, i.e., $Dl$ is an exact functor for this new $t$-structure (see \cite[1.3.16]{MR751966}), and so it takes the heart $\cH_{{\cD}}$ inside $\cH_{{\cC}}$. This shows that the restriction of $Dl$ to $\cH_{\cD}$ defines a functor $\xymatrix{\cH_{\cD} \ar[r]^{l_\cH}  &\cH_{\cC}}$, which is clearly exact.

On the other hand, the fact that $i$ is left exact ensures that $Ri$ takes $t(\cF)$ inside $t(\cY)$. 
Let $\tau_{t(\cX)} \colon D(\cD)\to t(\cX)$ be the right adjoint of the inclusion 
$t(\cX)\to D(\cD)$ (see \cite[Proposition~1.3.3.(i)]{MR751966}). Then the restriction of the composition $\tau_{t(\cX)}\circ Ri$ to $\cH_{\cC}$ gives a functor $i_\cH \colon \cH_{\cC}\to \cH_{\cD}$ and it is easy to see that $l_\cH$ is left adjoint of $i_\cH$ by composing the previous adjunctions.

Next, using Lemma~\ref{DlComm} we have that

%Let $D^\point\in D(\cC)$ and 
%\begin{equation}\label{triaD}
%\xymatrix{ \tau_{t(\cX)}(D^\point) \ar[r]& D^\point \ar[r]& \tau_{t(\cY)}(D^\point) \ar[r]^{\qquad+1}& 
%}
%\end{equation}
%its distinguished triangle, with $\tau_{t(\cX)}(D^\point) \in t(\cX)$ and $\tau_{t(\cY)}(D^\point) \in t(\cY)$. By applying the functor $Dl$ to (\ref{triaD}) we get the triangle in $D(\cC)$
%\begin{equation}\label{triaC}
%\xymatrix{ Dl(\tau_{t(\cX)}(D^\point)) \ar[r]& Dl(D^\point) \ar[r]& Dl(\tau_{t(\cY)}(D^\point)) \ar[r]^{\qquad+1}& 
%}
%\end{equation}
%with $Dl(\tau_{t(\cX)}(D^\point)) \in t(\cT)$ and $Dl(\tau_{t(\cY)}(D^\point)) \in t(\cF)$, so (\ref{triaC}) is the distinguished triangle associated to $Dl(D^\point)$, which proves that $Dl\circ\tau_{t(\cX)}=\tau_{t(\cT)}\circ Dl$. Thus

$$\begin{matrix}
l_\cH \circ i_\cH =& Dl \circ \tau_{t(\cX)}\circ Ri_{|\cH_{\cC}} \cong& 
\tau_{t(\cT)}\circ Dl \circ Ri_{|\cH_{\cC}}\hfill \\
\hfill\cong& \tau_{t(\cT)}\circ D(l \circ i)_{|\cH_{\cC}}\cong& \tau_{t(\cT)}\circ id_{\cH_{\cC}}  \hfill \\
\hfill\cong& id_{\cH_\cC}\hfill \\
\end{matrix}
$$
 and from this we conclude that $i_\cH$ is fully faithful.
 
Finally, 
\[
i_\cH \circ l_\cH (\cX[0]) \subseteq \tau_{t(\cX)}\circ (Ri \circ Dl)(D^{\geq0}(\cD)) \subseteq
\tau_{t(\cX)}(D^{\geq0}(\cD)) \subseteq \cX[0].
\]

%$$\begin{matrix}
%i_\cH \circ l_\cH (\cX[0])=& \tau_{t(\cX)}\circ Ri \circ Dl (t(\cX)\cap D^{\geq0}(\cD)) \hfill \\
%\hfill\subseteq& \tau_{t(\cX)} \circ Ri (t(\cX)\cap D^{\geq0}(\cC)) \hfill \\
%\hfill \subseteq& \tau_{t(\cX)}(D^{\geq0}(\cD))\hfill \\
%\hfill \subseteq& t(\cX)\cap D^{\geq0}(\cD)=\cX[0]. \hfill
%\end{matrix}
%$$
\end{proof}
%%%
\smallskip

%%%
\begin{remark}
Let us explain two examples in which one can apply the previous result. 
As a first example let $\cC$ be an abelian category satisfying $AB4^*$ (that is, small products exist in $\cC$ and such products are exact in $\cC$) and with enough injectives, and 
$i \colon \cC \to \cD$ is an additive functor. Then the right derived functor $Ri \colon D(\cC) \to D(\cD)$ exists by \cite[APPLICATION 2.4]{MR1214458}. 
Another interesting case is the one in which the category  $\cC$ admits enough $i$-acyclic objects. 
In this case one can use the same argument as in Proposition~\ref{adjhearts} restricted to the bounded below derived categories in order to obtain the same result.
\end{remark}
%%%
\smallskip
%%%

Dually, we have:
%%%
\begin{theorem}\label{cadjhearts}
Let  $\cD$  be an abelian category with a distinguished co-Giraud subcategory $\cC$ such that 
$j$ admits a left derived functor ${Lj}$. 
Let $(\cX,\cY)$ a torsion pair on $\cD$ such that $jr(\cX)\subseteq\cX$, and let $(\cT,\cF)=(r(\cX),r(\cY))$ be the induced torsion pair on $\cC$. Let us denote by $\cH_{\cC}$ and $\cH_{\cD}$ the associated hearts. 
Then there exists a distinguished co-Giraud subcategory
$(\cH_\cD,\cH_\cC,r_\cH,j_\cH)$ such that 
$j_\cH(r_\cH(\cY[1]))\subseteq \cY[1]$.
\end{theorem}
%%%
\smallskip

The next result shows that the contexts described in Theorems~\ref{adjhearts} and \ref{cadjhearts} are as general as possible.

\begin{theorem}\label{reconstruction}
Let $\cD$ be an abelian category endowed with a torsion pair
$(\cX,\cY)$ and let $\cH_{\cD}$ be the corresponding heart with respect to the
$t$-structure on $D(\cD)$ induced by $(\cX,\cY)$.
Let $\cS'$ be a Serre subcategory of $\cH_\cD$ and
$l':\cH_\cD\to \cC':=\cH_\cD/\cS'$ be its corresponding quotient functor,
and let us suppose that $i'l'(\cX[0])\subseteq X[0]$ ( i.e., $(l'(\cY[1]),l'(\cX[0]))$ is a torsion pair on $\cC'$).
%Let us suppose given a distinguished Giraud subcategory
%$(\cH_{\cD},\cC',l',i')$
%such that
%$i'(l'(\cX[0]))\subseteq \cX[0]$.
Then:
\begin{enumerate}
%\item the class $\cN=\{D^\point \; | \; l_\cH(H^i_t(D^\point))=0\; \forall i \in \Bbb Z\}$ defines a thick
%subcategory of $D(\cD)$.
\item  The class 
$\cS=\{ D\in \cD\; | \; l'(H_t^i(D))=0 \; \forall i\in \Bbb Z\}$
is a Serre subcategory 
%(hereditary torsion class) 
of $\cD$.
\item Denoted by $\cC:=\cD/\cS$ the quotient category and by
$l:\cD\ra \cC$ the quotient functor, then $l$ is exact 
%and 
%there exists a canonical functor
%$\phi:D(\cD)/\cN \ra D(\cC)$ which is an equivalence of categories. 
%\item T
and the classes $(l(\cX),l(\cY))$ define a torsion pair on $\cC$.
\item There is an equivalence of categories
$\cC' \buildrel{\cong}\over\rightarrow\cH_\cC$ 
for whom  $l_\cH$ (defined in \ref{adjhearts}) is identified with $l'$. 
\item 
Moreover in the case in which  the torsion-free class $\cY$ generates  $\cD$
(or dually if the torsion class $\cX$ cogenerates  $\cD$) and
if $(\cH_\cD,\cC',l',i')$ is a distinguished  Giraud (resp. $(\cH_\cD,\cC',l',j')$ co-Giraud) 
subcategory such that $i'$ (resp. $j'$) admits  a derived functor,
then
the functor $l$ admits a right adjoint $i$ (resp. left adjoint $j$)
such that the $(\cD,\cC,l,i)$ is a  
distinguished Giraud subcategory of $\cD$ which induces the distinguished Giraud
(resp. co-Giraud)
subcategory $\cC'$
of $\cH_\cD$.
\end{enumerate}
\end{theorem}
\begin{proof}

{$1.$}
%Let consider
%$\ra D_1^\point \ra D_2^\point \ra D_3^\point \buildrel{+1}\over\ra D_1^\point[1]\ra$ a 
%distinguished triangle in $D(\cD)$ such that $D_1^\point$, $D_2^\point\in \cN$.
%It induces 
%a long exact sequence of cohomology 
%in $\cH_\cD$ 
%\begin{equation*}
%\cdots H_t^{i-1}(D_3^\point)\ra H_t^i(D_1^\point)
%\ra H_t^i(D_2^\point)\ra H_t^i(D_3^\point)\ra 
%H_t^{i+1}(D_1^\point)\ra \cdots
%\end{equation*}
%where $H_t^i$ denotes the $i$-th cohomology with respect to
% the $t$-structure $(t(\cX),t(\cY))$ (see \ref{Tstr}). 
%Now by applying the functor $l_\cH$ (which is exact) we obtain the long exact sequence
%\begin{equation*}
%\scriptstyle{
%\cdots l_\cH(H_t^{i-1}(D_3^\point))\ra l_\cH(H_t^i(D_1^\point))
%\ra l_\cH(H_t^i(D_2^\point))\ra l_\cH(H_t^i(D_3^\point))\ra 
%l_\cH(H_t^{i+1}(D_1^\point))\ra \cdots}
%\end{equation*}
%in $\cC'$ which proves that $D_3^\point\in \cN$.
%\smallskip 
%{\cal 2.}
We have to prove that given a short exact sequence $0\ra S_1\ra S\ra S_2\ra 0$ in $\cD$
the middle term $S$ belongs to $\cS$ if and only if $S_1, S_2\in \cS$ where $\cS$
is defined as $\cS=\{ D\in \cD\; | \; l'(H_t^i(D))=0 \; \forall i\in \Bbb Z\}$.
Now, any short exact sequence on $\cD$ defines a distinguished triangle in $D(\cD)$ and so
one obtain the long exact sequence in $\cH_\cD$
\begin{equation}\label{LES}
\scriptstyle{\cdots H_t^{-1}(S_2)\ra H_t^0(S_1)\ra H_t^0(S)\ra H_t^0(S_2)\ra 
H_t^1(S_1)\ra H_t^1(S)\ra H_t^1(S_2)\ra H_t^2(S_1)\cdots}
\end{equation}
By \ref{RTS},
$H_t^{-1}(S_2)=0=H_t^2(S_1)$ and for any $D\in\cD$ one has
$H^0(D)=t(D)[0]$ as a complex concentrated in degree $0$  while 
$H^1(D)={D\over{t(D)}}[1]$.
So the sequence (\ref{LES}) reduces to the sequence in $\cH_\cD$
\begin{equation}\label{LES1}
0\ra t(S_1)[0]\ra t(S)[0]\ra t(S_2)[0]\ra 
{S_1\over{t(S_1)}}[1]\ra {S\over{t(S)}}[1]\ra {S_2\over{t(S_2)}}[1]\ra 0.
\end{equation}

Let us recall that the class 
\begin{equation}\label{SerreSC}
\cS'=\{E\in \cH_\cD\; | \; l'(E)=0\}
\end{equation}
is a Serre subcategory of  $\cH_\cD$. 
So from one side it is clear that if $S_1, S_2\in\cS$ then
$t(S_i)[0], {S_i\over{t(S_i)}}[1]\in \cS'$ for any $i\in \{1,2\}$,
which implies that $t(S)[0]$  and $\frac{S}{t(S)}[1]$ belong to $\cS'$,
and so $S\in \cS$.

On the other side if $S\in\cS$ then $t(S)[0], {S\over{t(S)}}[1]\in \cS'$,
%
%which implies that also
%$S_2\cong t(S_2)$ (since $\cX$ is closed under quotients) so (\ref{LES1})
%becomes the exact sequence $0\ra t(S_1)\ra S \ra S_2 \ra {S_1\over{t(S_1)}}[1]\ra 0$
%in $\cH_D$.
and by applying the functor $l'$ (which is exact by hypothesis) to (\ref{LES}) 
we obtain the exact sequence in $\cC'$
$$0\ra l'(t(S_1)[0])\ra 0 \ra l'(t(S_2)[0]) \ra l'\left({S_1\over{t(S_1)}}[1]\right)\ra 0\ra
 l'\left({S_2\over{t(S_2)}}[1]\right)\ra 0.
$$
This proves that $t(S_1)[0],{S_2\over{t(S_2)}}[1]\in \cS'$ and 
$ l'(t(S_2)[0])\cong l'\left({S_1\over{t(S_1)}}[1]\right)\in
l'(\cX[0])\cap l'(\cY[1])=0
$
which proves that 
$t(S_2)[0], {S_1\over{t(S_1)}}[1]\in \cS'$ 
and so $S_2\in\cS$ and $S_1\in \cS$.

%\smallskip 

%{\cal 3.}
%The class $\cS$ is a Serre subcategory of $\cD$ then it is possible to define
%the quotient category $\cD/\cS$ (see \cite{???}) whose objects are
%the objects of $\cD$ while morphisms in $\cD/\cS$ are defined as
%\[
%{\cD/\cS}\,(D_1,D_2) = \varinjlim_{D_1',D_2'}{\cD}\,(D_1',D_2/D_2')
%\]
%where $D_1'$ runs through the subobjects of $D_1$ such that $D_1/D_1' \in\cS$ and
%$D_2'$ runs through the subobjects of $D_2$ which belong to $\cS$.

%

%%while morphisms in $\cD/\cS$ are obtained by those of $\cD$
%%by formally inverting any morphism $f:D_1\ra D_2$ in $\cD$ such that
%%${\Ker}(f), {\Coker}(f)\in \cS$.

%\smallskip

%{\cal 4.}

\smallskip

{$2.$}
Let us show that the classes $(l(\cX),l(\cY))$ define a torsion pair on $\cC$. First of all, since any object of $\cC$ may be regarded as an object of $\cD$ and the functor $l$ is exact, it is clear that any object $C\in\cC$ is the middle term of a short exact sequence $0 \to X \to C \to Y \to 0$ with $X\in l(\cX)$ and $Y\in l(\cY)$. It remains to show that $\cC(X,Y)=0$,for every $X\in \cX$ and every $Y\in \cY$. So let  $X\in l(\cX)$ and $Y\in l(\cY)$. A morphism $\varphi\colon X \to Y$ in $\cC$ may be viewed as the class of a morphism $X' \to Y/Y'$ in $\cD$, where $X/X'$ and $Y'$ are in $\cS$. Let $t(X')$ be the torsion part of $X'$ (viewed as an object of $\cD$) with respect to the torsion pair $(\cX,\cY)$ in $\cD$ and $Y/Y''$ be the torsion-free quotient of $Y/Y'$. We show that the composite morphism $t(X')\to X' \to Y/Y' \to Y/Y''$ also represents the morphism $\varphi$ 
in $\cC$, i.e., $X/t(X')\in \cS$ and $Y''\in\cS$. Hence $\varphi=0$, since it is a morphism from a torsion to a torsion-free object. Now, the short exact sequence in $\cD$
\[
0 \to {X'\over t(X')} \to {X\over t(X')} \to {X\over X'} \to 0
\]
defines a distinguished triangle in $D(\cD)$ and so one obtains the long exact sequence of cohomology in $\cH_\cD$
\begin{equation*}
\textstyle
{\cdots\ H_t^0\left({X'\over t(X')}\right)\ra H_t^0\left({X\over t(X')}\right)\ra H_t^0\left({X\over X'}\right)\ra 
H_t^1\left({X'\over t(X')}\right)\ra H_t^1\left({X\over t(X')}\right)\ \cdots}
\end{equation*}
which reduces to 
\begin{equation*}
0\ra {X\over t(X')}[0]\ra {X\over X'}[0]\ra {X'\over t(X')}[1]\ra 0
\end{equation*}
since $H_t^0({X'\over t(X')})=t\left(X'\over t(X')\right)[0]=0$ and 
$H_t^1({X\over t(X')}) = {{X/t(X')}\over{t(X/t(X'))}}[1]=0$ (since $X\in\cX$).
By applying the exact functor $l'$ we obtain the exact sequence in $\cC'$
\begin{equation*}
0\ra l'\left({X\over t(X')}[0]\right)\ra l'\left({X\over X'}[0]\right)\ra l'\left({X'\over t(X')}[1]\right) \ra 0
\end{equation*}
where $l'\left({X\over X'}[0]\right)=0$ because $X/X' \in \cS$. 
Hence $l'\left({X\over t(X')}[0]\right)=0=l'\left({X'\over t(X')}[1]\right)$. 
In particular, $X/t(X') \in \cS$. A dual argument shows that $Y'' \in \cS$.

%%%%
\smallskip

{$3.$} Given a distinguished Giraud subcategory
$(\cH_{\cD},\cC',l',i')$, one can identify $\cC'$ with the quotient category of
$\cH_\cD$ with respect to its Serre subcategory $\cS'$ defined in \ref{SerreSC}.
Applying Proposition~\ref{SonH} we see that the functor $l$ previously defined induces an exact essentially surjective functor
$l_{\cH}:\cH_\cD\rightarrow \cH_\cC$, and this proves that $\cH_\cC\cong \cH_\cD/\cS_\cH$
where $\cS_{\cH}$ is the kernel of the functor $l_\cH$.
In order to conclude the proof of this third statement it is enough to prove that 
$\cS_{\cH}$  coincides with $\cS'$.

An object $X^{-1}\overset{x}{\longrightarrow }X^{0}$ in $\cH_\cD$ is in the kernel of $l_\cH$ if and only if the complex $l(X^{-1})\overset{l(x)}{\longrightarrow }l(X^{0})$ is zero in $\cH_\cC$, that is:  
$\Ker(l(x))=l(\Ker(x))=0$ and $\Coker(l(x))=l(\Coker(x))=0$. 
This proves that $\Ker(x) \in \cS\cap\cY$ which is equivalent to $\Ker(x)[1]=H^1_t(\Ker(x))\in \cS'$,
and $\Coker(x) \in \cS\cap\cX$  which is equivalent to $\Coker(x)[0]=H^0_t(\Coker(x))\in \cS'$.
So $X^{-1}\overset{x}{\longrightarrow }X^{0}$ belongs to $\cS'$.

\smallskip

{$4.$} 
Let us suppose that the torsion-free class $\cY$ generates  $\cD$.
Then it is clear that $l(\cY)$ generates the quotient category $\cD/\cS$
and so by \cite[Theorem 8.2]{MR2486794}
the double heart $\cH_{\cH_{\cD}}$ is equivalent to $\cD$ and
$\cH_{\cH_{\cC}}\cong \cC$.

If, moreover, $(\cH_\cD,\cC',l',i')$ is a distinguished Giraud  subcategory 
such that $i'$  admits  a derived functor,
then we can apply
Theorem~\ref{adjhearts} in order to obtain a distinguished Giraud  subcategory  on the associated hearts.
This proves that
the functor $l\cong l_{l_{\cH}}$ admits a right adjoint $i$ 
such that $(\cD,\cC,l,i)$ is a  
distinguished Giraud subcategory of $\cD$ which induces the distinguished Giraud
subcategory $\cC'$
of $\cH_\cD$.
\end{proof}
\smallskip
%%%

\bibliographystyle{amsplain}
\bibliography{bibart}

\end{document}